\documentclass[letter]{amsart}

\usepackage[utf8]{inputenc}
\usepackage[english]{babel}
\usepackage{amsmath,amssymb,amsthm,mathrsfs}
\usepackage[colorlinks,linktocpage]{hyperref}
\usepackage[all]{xy}
\usepackage{import}
\usepackage{pdfpages}
\usepackage{transparent}

\usepackage{outlines}
\usepackage{mdframed}
\usepackage{enumitem}

\def\diam{\text{diam}}
\def\dim{\text{dim}}

\def\Vol{\text{Vol}}

\def\inj{\text{inj}}

\newtheorem{thm}{Theorem}[section]

\newtheorem{cor}[thm]{Corollary}
\newtheorem{prop}[thm]{Proposition}

\newtheorem*{claim*}{Claim}

\theoremstyle{definition}

\newtheorem{exmp}{Example}[section]

\theoremstyle{remark}
\newtheorem{rmk}{Remark}[section]
\newtheorem*{rmk*}{Remark}

\newtheorem*{fact*}{Fact}

\title[Three-manifolds with bounded curvature]{Three-manifolds with bounded curvature and uniformly positive scalar curvature}
\author{Conghan Dong}
\address{Mathematics Department, Stony Brook University, NY 11790, United States}
\email{conghan.dong@stonybrook.edu}
\begin{document}
\date{\today}
\begin{abstract}
	In this note, we prove that for a complete noncompact three dimensional Riemannian manifold with bounded sectional curvature, if it has uniformly positive scalar curvature, then there is a uniform lower bound on the injectivity radius.
\end{abstract}
\keywords{$3$-manifolds, positive scalar curvature, collapsing, injectivity radius}
\maketitle

\section{Introduction}
A well-known problem posed by Yau (see Problem Section in \cite{Yau82}) is how to classify $3$-manifolds admitting complete Riemannian metrics of positive scalar curvature up to diffeomorphism. In the case of closed $3$-manifolds, it has been resolved by Schoen-Yau \cite{SchoenYau79}, Gromov-Lawson \cite{GromovLawson83}, Hamilton \cite{Hamilton82} and Perelman \cite{Perelman2002entropy, Perelman2003surgery, Perelman2003finite}. But for open $3$-manifolds, this problem remains wide open. Some recent progress were made under the assumption of uniformly positive scalar curvature. Particularly, Chang-Weinberger-Yu \cite{CWY10} classified complete $3$-manifolds with uniformly positive scalar curvature and finitely generated fundamental group. Bessieres-Besson-Maillot \cite{BBM11} classified complete $3$-manifolds with uniformly positive scalar curvature and bounded geometry. Here the bounded geometry means the sectional curvature is bounded and the injectivity radius is bounded away from zero. 

In this note, we make further progress towards Yau's problem and show that bounded curvature and uniformly positive scalar curvature on complete open $3$-manifolds are enough to derive a lower bound on the injectivity radius.

\begin{thm}\label{main-thm}
	Assume $(M^3, g)$ is a smooth three dimensional complete noncompact Riemannian manifold. If the sectional curvature is bounded by $|K_g| \leq \Lambda $ and the scalar curvature is bounded below by $R_g \geq 1$, then the injectivity radius has a uniform lower bound. In other words, there exists a uniform positive number $r_0=r_0(\Lambda ) >0$, depending only on $\Lambda $, such that for any $x \in M$, the injectivity radius at $x$ is bounded from below by $r_0$.
\end{thm}

This theorem together with the main result in \cite{BBM11} gives the following corollary.

\begin{cor}
	Let $M$ be a connected orientable $3$-manifold which carries a complete metric of bounded sectional curvature and uniformly positive scalar curvature, then there is a collection $\mathcal{F}$ of spherical space forms with finitely many diffeomorphim type such that $M$ is a (possibly infinite) connected sum of copies of $S^2 \times S^1$ with members of $\mathcal{F}$.
\end{cor}
\begin{rmk}
	Jian Wang recently claimed a result on classification of open three manifolds carrying uniformly positive scalar curvature in \cite{Wang22}.
\end{rmk}

Also, we have the following improvement of the main result in \cite{BBMM20}.

\begin{cor}
	The moduli space of complete Riemannian metrics of bounded curvature
and uniformly positive scalar curvature on an orientable $3$-manifold is path-connected or empty.
\end{cor}

By regularity results in \cite{Anderson90}, the main theorem implies the following compactness result.

\begin{cor}
	For any sequence of pointed complete noncompact Riemannian $3$-manifolds $(M_i^3, g_i, p_i)$ with uniformly bounded curvature and uniformly positive scalar curvature, up to a subsequence, they converge locally in $C^{1, \alpha }$-topology for any $0<\alpha <1.$
\end{cor}

Another corollary is that $3$-manifolds with uniformly positive scalar curvature and bounded curvature have at least linear volume growth.

\begin{cor}
	Assume $(M^3, g)$ is a complete noncompact $3$-manifold satisfying that the sectional curvature is bounded by $|K_g| \leq \Lambda $ and the scalar curvature is bounded below by $R_g \geq 1$. Then there is a positive constant $c=c(\Lambda )>0$, depending only on $\Lambda $, such that for any point $p \in M$, 
	$$
	\liminf_{r\to \infty} \frac{\Vol_gB(p, r) }{r} \geq c.
	$$ 
\end{cor}

\begin{proof}
	For any $r \gg 1$, and $q \in \partial B(p, r)$, choose a geodesic segment $\gamma: [0, r]\to M $ between $p$ and $q$. Let $r_0$ be the lower bound on the injectivity radius given in Theorem \ref{main-thm}. Choose a partition of the interval $[0, r]$ by $0=t_0 < t_1 < \cdots < t_{k+1} \leq r$ with $t_{i+1}=t_i + 3r_0$ for $0\leq i \leq k$ and $r< t_{k+1}+ 3r_0$. Then the geodesic balls $\{ B( \gamma (t_i), r_0)\} _{i=0}^k$ are disjoint and all included in $B(p, r)$. Note that $k \geq \frac{r}{3r_0}$.

	For any point $x \in M$, since the injective radius at $x$ is bigger than $r_0$ and $K_g \leq \Lambda $, for $r_1:= \frac{1}{2}\min \{ r_0, \frac{\pi}{ \sqrt{\Lambda } }\} $, we have a polar coordinate such that $g= dr^2+ g_r$ in $B(x, r_1)$. By the Hessian comparison theorem (\cite[Chapter 6, Corollary 2.4]{Petersen98}), in $B(x, r_1)$ we have
	$$
	(g_r(\partial_i, \partial_j) ) \geq \frac{1}{\Lambda }\sin ^2 \sqrt{\Lambda } r \cdot I_2
	$$
	as $2 \times 2$-matrix, where $\partial _i = d\exp _x(\partial \theta _i)$ for a coordinate $\{\theta _i\} $ of the unit sphere in $T_qM$. Denote by $g_\Lambda = dr^2 + \frac{1}{\Lambda }\sin ^2 \sqrt{\Lambda } r ds_2^2$ the standard metric on the $3$-sphere with constant curvature $\Lambda $.
	So $$\Vol_g B(x, r_0) \geq \Vol_g B(x, r_1)\geq \Vol_{g_\Lambda }B(r_1)=: v_{\Lambda }(r_0).$$ 
	
	Together with above arguments, we have $$
	\Vol_g B(p, r) \geq \sum_{i=0}^k \Vol_g B(\gamma (t_i), r_0) \geq r\cdot \frac{v_\Lambda (r_0)}{3r_0}.
	$$ 
	It's done by taking $c= \frac{v_\Lambda (r_0)}{3r_0}$.
\end{proof}

\begin{rmk}
	Yau \cite{Yau76} proved that any complete noncompact manifolds with nonnegative Ricci curvature have at least linear volume growth. In general, the volume growth is not uniform. Munteanu-Wang \cite{MunteanuWang22} proved that for a complete $3$-manifold with nonnegative Ricci curvature and uniformly positive scalar curvature, the volume has at most uniform linear growth. So in the case of a complete noncompact $3$-manifold $(M^3, g)$ with $0 \leq Ric_g \leq \Lambda $ and $R_g \geq 1$, the volume of a geodesic ball has uniformly linear growth. That is, for some uniform constants $c_1,c_2>0$, as $r\to \infty$,
	$$
	c_1 r \leq \Vol_gB(p,r) \leq c_2 r.
	$$ 
\end{rmk}

Now we explain the ideas in the proof of the main theorem and give some remarks. The proof is by contradiction, and involves two main ingredients: the structure of collapsing manifolds with bounded curvature, which has been intensively studied, see works \cite{CheegerGromov86, Fukaya87, Fukaya88, CFG92}; and Gromov's width inequality \cite{Gromov18, Gromov19}.

Let us assume by contradiction that, up to a double covering of orientations, there is a sequence of three dimensional complete noncompact orientable Riemannian manifolds $(M_i^3, g_i, p_i)$ with uniformly bounded curvature and uniformly positive scalar curvature, which, up to a subsequence, converges to a complete Alexandrov space $(X, d, o)$ in the pointed Gromov-Hausdorff topology, and the injectivity radius $\inj(p_i)$ at $p_i$ converges to $0$. 

By our assumptions, the sequence is collapsing and $X$ can't be a point, i.e. $\dim X = 1$ or $2$. The theory of collapsing with bounded curvature gives a symmetric structure around sufficiently collapsed part. Roughly speaking, if $\dim X=2$, there is a wide Seifert fibered space with boundary around $p_i$ in $M_i$, see Proposition \ref{seifert}; if $\dim X=1$, there is a long torical band in $M_i$, see Proposition \ref{dimX=1}. The noncompact assumption ensures that we can find a Seifert fibered space with any arbitrary
width or a torical band whose length is large enough.

But Gromov's width inequality tells us that uniformly positive scalar curvature on a Riemannian band imposes a uniform upper bound on the width, except that there are spheres separating this band, see Propositions \ref{torical-band} and \ref{3-dim-width-ineq}. When $\dim X=1$, this immediately gives a contradiction. When $\dim X=2$, note that a Seifert fibered space with boundary is $S^2$-irreducible, so spheres can not separate the wide Seifert fibered space in the sufficiently collapsed part, which also gives a contradiction (see Section \ref{main-proof} for more details).

\begin{rmk}
	The following example shows that the assumption on the upper bound of sectional curvature is necessary. 
\end{rmk}

\begin{exmp}
	Consider the warped product metric $dr^2 + \rho (r)^2 ds_2^2$ on $\mathbb{R}\times S^2$. Then the scalar curvature is given by 
	$$
	R= -\frac{4\ddot{\rho}}{\rho } + \frac{2-2\dot{\rho} ^2}{\rho ^2},
	$$ 
	and the sectional curvature $K$ is between the values $-\frac{\ddot{\rho }}{\rho }, \frac{1-\dot{\rho }^2}{\rho^2 }$. See \cite{Petersen98} for computations. Define
	$$\rho (r)=
	\begin{cases}
		f(r), & |r| \leq R_0\\
		\frac{1}{|r|}, & |r| \geq R_0,
	\end{cases}
	$$
	where $$f(r)= \frac{3}{8R_0^5}r^4 - \frac{5}{4R_0^3}r^2 + \frac{15}{8R_0}.$$ Then $\rho $ is a positive $C^2$-function, and for a fixed $R_0 \geq 100$, it's easy to see that $R >1$, and $K \geq -\frac{C}{R_0^2}$. A smooth modification of $\rho $ gives a metric with sectional curvature bounded from below and uniformly positive scalar curvature on $\mathbb{R}\times S^2$. But the injectivity radius converges to $0$ and the curvature upper bound blows up as $|r|\to \infty$.
\end{exmp}

\begin{rmk}
	The main theorem still holds for closed $3$-manifolds with big diameter by the same proof. That is, there exists a uniform constant $D_0>0$, coming from Gromov's width inequality (e.g. $D_0=\frac{4 \sqrt{6} \pi}{3}+1$) , such that 
if $(M^3, g)$ is a closed $3$-manifold satisfying $|K_g| \leq \Lambda $, $R_g \geq 1$ and $\diam(M, g) \geq D_0$, then there exists a uniform positive number $r_0= r_0(\Lambda )$ depending only on $\Lambda $ such that the injectivity radius of $(M,g)$ is bounded from below by $r_0$. 

	In general, if we drop the diameter assumption, there exists closed $3$-manifolds with uniformly bounded sectional curvature and uniformly positive scalar curvature, but arbitrarily small injectivity radius. Some examples are given by product metrics $ds_2^2 + \varepsilon ^2 ds_1^2$ on $S^2 \times S^1$ with small $\varepsilon >0$, and collapsing Berger $3$-spheres.

\end{rmk}

\begin{rmk}
	The main result does not hold in higher dimensions. For example, we can consider $S^2 \times \mathbb{R}^2$ with product metric  $g=g_{S^2}+g_{\mathbb{R}^2}$, where $(\mathbb{R}^2, g_{\mathbb{R}^2})$ is a complete metric with bounded curvature, and the injectivity radius goes to zero at infinity in $\mathbb{R}^2$, and $(S^2, g_{S^2})$ is a constant curvature metric. We can choose $g_{S^2}$ with big enough curvature such that $g$ has uniformly positive scalar curvature and bounded sectional curvature, but the injectivity radius does not have a uniform lower bound.
\end{rmk}

\textbf{Acknowledgements.} The author would like to thank his advisor Prof. Xiuxiong Chen for the encouragements. The author thanks Jian Wang and Shaosai Huang for careful reading and helpful comments, and Jian Wang for helpful discussions on positive scalar curvature and informing the author some corollaries. He also thanks the referee for helpful comments and suggestions.

\section{Collapsing with bounded curvature}
In this section, following \cite{Fukaya87, Fukaya88}, we talk about the structure around sufficiently collapsing part of three manifolds with bounded curvature.

Assume there is a sequence of complete noncompact orientable three manifolds $(M_i, g_i, p_i)$ with sectional curvature uniformly bounded $| K_{g_i}| \leq \Lambda $ and injectivity radius $\inj(p_i)$ satisfying $$
\lim_{i\to \infty} \inj(p_i)=0.
$$ 

By Cheeger-Gromov compactness theorem \cite{Cheeger70, Gromov07}, we can take a convergent subsequence in the pointed Gromov-Hausdorff topology. For simplicity, we abuse the notation and write $(M_i, d_i, p_i)$ for a convergent subsequence, that is $$(M_i, d_i, p_i) \to (X, d, o),$$ where $(X,d)$ is a complete Alexandrov space with curvature bounded below and $0\leq \dim X\leq 3$. By results in \cite{Cheeger70, CheegerColding97}, the assumption that $\inj(p_i)\to 0$ implies the sequence is collapsing, i.e. $\dim X<3$; the noncompactness assumption implies that $\dim X >0$. Namely, $\dim X=1$ or $2$.

\begin{prop}\label{dimX=2}
	If $\dim X=2$, then $(X,d)$ is a Riemannian orbifold without boundary. 
\end{prop}

\begin{rmk}
	If $X$ has bounded diameter and comes from codimension one collapsing, then the fact that it is a Riemannian orbifold was known in the literature, see Proposition 11.5 in  \cite{Fukaya90}, and Proposition 8.1 in \cite{ShioyaYamaguchi00}. See also \cite{NaberTian08} for the general collapsing case. For reader's convenience, we reproduce the proof and add some details here. 
\end{rmk}

\begin{proof}

	For any $q_0 \in X$, take $q_i\in M_i$ with $q_i \to q_0$. For a small ball $B(0, \varepsilon )\subset \mathbb{R}^3 \cong  T_{q_i}M_i$ with center $0$ and a very small radius $\varepsilon \ll\frac{\pi}{ \sqrt{\Lambda} }$, consider the pull back metrics $\exp _{q_i}^{*}g_i$, which we still denote by $g_i$, then $(B(0, \varepsilon ), g_i)$ has a uniform lower bound on injectivity radius. By Cheeger-Gromov compactness theorem, up to a subsequence we have a metric $g_0$ on $B(0, \varepsilon )$ such that $$(B(0, \varepsilon ), g_i) \to (B(0, \varepsilon ), g_0)$$ in $C^{1,\alpha }$-topology. Take $G_i$ to be the local fundamental group $$G_i= G(q_i, \varepsilon ):= \{ \gamma : \gamma \text{ is a geodesic loop at } q_i \text{ with length } < \varepsilon \} .$$ 
	Then $G_i$ is a finite set and acts on $B(0, \varepsilon )$  by free local isometries and $B(0, \varepsilon ) / G_i = B(q_i, \varepsilon )$. The local group structure $(G_i, *)$ is defined as follows: for $\gamma _1, \gamma _2, \gamma _3 \in G_i$, we put $\gamma _1 * \gamma _2 = \gamma _3$ if $\gamma _1 * \gamma _2$ is well defined and coincides with $\gamma _3$, where $\gamma _1*\gamma _2$ is the Gromov's product defined as the unique geodesic loop in the short homotopy class of $\gamma _1 \cdot \gamma _2$, see \cite{BK81} for more details. Put the set of maps
	$$
	L= \{ \gamma  \in C( B(0, \varepsilon ), B(0, 2\varepsilon ) ): \frac{1}{2}\leq \frac{d_0(\gamma (x), \gamma (y) )}{d_0(x, y)} \leq 2,\ \forall x, y \in B(0, \varepsilon )\} ,
	$$ where $d_0$ is the metric associated to $g_0$. Note that by Arzel\`a-Ascoli theorem, $L$ is a compact set. For large enough $i$, we have $G_i \subset L$. By taking a subsequence, there exists a closed subset $G \subset L$ such that $G_i \to G$, and $G$ acts isometrically on $(B(0, \varepsilon ), g_0)$. It was proved in Section 3 of \cite{Fukaya88} that $G$ is a Lie group germ, which means that $G$ is locally isomorphic to a Lie group and its action on $B(0, \varepsilon )$ is smooth.

	Then passing to a subsequence, there is an equivariant convergent sequence $$(B(0, \varepsilon ), g_i, G_i) \to (B(0, \varepsilon ), g_0, G),$$ and $$(B(0, \varepsilon ), g_0) / G= B(q_0, \varepsilon ).$$ 
	In the case $\dim X=2$, we know $\dim G=1$. 

	Let $H_0$ be the isotropy sub-local group of $G$ at $0$, that is
	$$
	H_0:= \{ \gamma \in G: \gamma (0) =0\} .
	$$ 
	Then $H_0$ is in fact a group. To show this fact, for any $\gamma _1, \gamma _2 \in H_0$, it's enough to prove $\gamma _1 * \gamma _2$ is well defined and lies in $H_0$. Assume $\gamma _1^i\to \gamma _1, \gamma _2^i\to \gamma _2$ with $\gamma _1^i, \gamma _2^i \in G_i$. Since $d_0(\gamma _1(0), 0)= d_0(\gamma _2(0), 0)=0$, we know $d_i(\gamma _1^i(0) , 0), d_i(\gamma _2^i(0) , 0) \to 0$. In particular, the total length of $\gamma _1^i$ and $\gamma _2^i$ is smaller than $\varepsilon $ for all large $i$. Then $\gamma _1^i * \gamma _2^i$ is well defined and lies in $G_i$. Moreover, $d_i(\gamma _1^i * \gamma _2^i (0), 0) \leq d_i(\gamma _1^i(0), 0) + d_i(\gamma _2^i(0), 0)$. By taking a subsequence, $\gamma _1^i * \gamma _2 ^i \to \gamma _1 * \gamma _2 \in H_0$. 

	Let $H_0',G'$ be the identity components of $H_0, G$ respectively. From Section 5 of \cite{Fukaya88}, we know for $\varepsilon $ small enough, $B(q_0, \varepsilon )$ is isometric to $B(0, \varepsilon ) / H_0G'$.

	\begin{claim*}
		$H_0$ is discrete.
	\end{claim*}
	\begin{proof}[Proof of the claim]
		We argue it by contradiction and assume that $H_0$ is not discrete. Then $H_0$ contains a one dimensional Lie group germ. Since $\dim G=1$, $H_0$ and $G$ has the same Lie group germ at identity. Hence there exists a small neighbourhood $U$ of the identity $e$ in $G$ such that $\forall \alpha \in U$, $\alpha (0)=0$. 

		Notice that there is a small $\delta _0>0$ such that $\{\alpha  \in G: d_0(0, \alpha (0) ) \leq  \delta _0\} \subset H_0$. If not, then there exist $\delta _j\to 0$ and a sequence $\alpha _j \in G$ such that $d_0(0,\alpha _j(0) ) \leq \delta _j $ but $\alpha _j(0) \neq 0$. Then up to a subsequence, assume $\alpha _j \to \alpha _\infty \in G$. So $\alpha _\infty(0) =0$. Note $\alpha _\infty^{-1}\cdot \alpha _j \in U$ for large $j$ implies $\alpha _\infty^{-1}\cdot \alpha _j(0)=0$. So $\alpha _j(0) = \alpha _\infty( \alpha _\infty^{-1}\cdot \alpha _j(0) )= \alpha _\infty(0)=0$, a contradiction. 

		Similarly, for any small $0<\delta _1 <\delta _0$, we have that for any $\gamma _i \in G_i$, if $d_i(0, \gamma _i(0))\leq  \delta _0$, then $d_i(0, \gamma _i(0) ) < \delta _1$ for all large $i \geq i_0$ depending on $\delta _0, \delta _1$. If not, then there is a sequence $\gamma _i$ satisfying $\delta _1 \leq d_i(0, \gamma _i(0) ) \leq  \delta _0$. Up to a subsequence, assume $\gamma _i \to \gamma _0 \in G$. Then $\delta _1 \leq d_0(0, \gamma _0(0) ) \leq \delta _0$. But we just proved that $d_0(0, \gamma _0(0) ) \leq \delta _0$ implies $\gamma _0(0)=0$, which is a contradiction. 

		Now consider the local group $$G_i(\delta _0):=\{\gamma \in G_i: d_i(0, \gamma (0) ) \leq \delta _0\} .$$ There are two cases:

		(1) $G_i(\delta _0)=\{e\} $. In this case, $B(q_0, \delta _0)= (B(0, \delta _0),g_0)$ has dimension three and is thus noncollapsing, a contradiction.

		(2) $G_i(\delta _0) \neq \{e\} $. Assume that there is a non-identity element $\gamma_i \in G_i(\delta _0)$. Denote the uniform lower bound of injectivity radius on $(B(0, \varepsilon ), g_i)$ by $r_0$, and choose $\delta _1: = \min \{\frac{1}{2}\delta _0, \frac{1}{2}r_0, \frac{\pi}{4 \sqrt{\Lambda } }\} $. Let $i_0$ be a large number depending on $\delta _0, \delta _1$ in previous paragraph and $i \geq i_0$. 
		Inductively, for any $n \in \mathbb{Z}$ and $n \geq 2$, assume $d_i(0, \gamma _i^{k}(0) ) \leq \delta _0 $ for all $1 \leq k \leq n-1$. Then $d_i(0, \gamma _i^{n-1}(0) ) \leq \delta _0$ implies that  $$
		d_i(0, \gamma _i^{n-1}(0) ) < \delta _1 \leq \frac{1}{2}\delta _0.
		$$ 
		By triangle inequality, we have
\begin{align*}
	d_i(0, \gamma _i^n(0) )&\leq d_i(0, \gamma _i^{n-1}(0) ) + d_i(\gamma _i^{n-1}(0), \gamma _i^n(0) )\\
			       &= d_i(0, \gamma _i^{n-1}(0) ) + d_i(0, \gamma _i(0) )\\
			       & \leq \frac{1}{2}\delta _0 + \frac{1}{2}\delta _0\\
			       &= \delta _0.
\end{align*}
So for any $n \in \mathbb{Z}_+$, $\gamma _i ^{n} \in G_i(\delta _0)$. But $G_i(\delta _0)$ is a finite set, so there exists an $n_i$ such that $\gamma_i ^{n_i}=e$. 

		Then we can use the technique of center of mass to show that $\gamma _i =e$. By our choice of $\delta _1$, from Section 8 of \cite{BK81}, there is a unique minimum point $q_c$ of  $P(x)= \frac{1}{2}\sum_{j=1}^{n_i}d_i(x, \gamma _i^j(0) )$ in $B(0, \delta _1)$. Since $\gamma _i$ is an isometry and $\gamma _i^{n_i}=e$, $P(q_c)=P(\gamma _i(q_c) )$. From the uniqueness of $q_c$, we know $\gamma _i(q_c) =q_c$. Note $\gamma _i$ behaves like deck transformation, so $\gamma _i$ has a fixed point only when $\gamma _i=e$. This gives a contradiction and completes the proof of the claim.
\end{proof}
Note the fact that $H_0$ is a closed subset of $L$ and thus compact. From the claim, we know $H_0$ is a finite group. Then the orbit $G\cdot 0 \cong G / H_0$ is one dimensional. By standard slice theorem, we know $B(q_0, \varepsilon )$ is isometric to $S_0 / H_0$, where $S_0 = \exp _0(V)$ for some small ball $V \subset \mathbb{R}^2 \cong (T_0(G\cdot 0))^\perp \subset T_0B(0, \varepsilon )$, and via $\exp _0$, $H_0$ acts isometrically on $V$. Note that $M_i$ are orientable and $G_i$ preserve orientations, so $G$ also preserves orientation, which implies that $H_0 \subset SO(2)$, i.e. $H_0$ is a finite cyclic group. This implies that $q_0$ is either a regular point or an interior Riemannian orbifold point.
\end{proof}

\begin{prop}\label{seifert}
	If $\dim X=2$, then for any fixed number $r>1$, for all sufficiently large $i$, there exists a Seifert fibered space $\Omega _i$ with smooth boundaries $\partial \Omega _i = \partial _- \sqcup \partial _+$ such that  $$
	d_{g_i}(\partial _-, \partial _+) \geq r.
	$$  
\end{prop}

\begin{proof}
	Under the same notations as above, if $\dim X=2$, then from Proposition \ref{dimX=2} we have a decomposition $S_1(X) \subset X$, where $S_1(X)$ consists of discrete orbifold singularities and $X\setminus S_1(X)$ is a smooth surface. 

	By \cite[Theorem 0.12, Theorem 10.1]{Fukaya88}, there is a continuous map $$f_i: B(p_i, 4r) \to X$$ such that $f_i(p_i)=o$ and 
	$$
	|d_i(x, y) - d(f_i(x), f_i(y) )| \leq C(1+r) \varepsilon _i,
	$$ where $C>0$ is a uniform constant and $\varepsilon_i\to 0$ as $i\to \infty$. Moreover,
	the restriction of $f_i$ on $f_i^{-1}(X\setminus S_1(X) )$ is a smooth fiber bundle with fiber diffeomorphic to $S^1$, and $\forall p \in S_1(X) \cap B(o, 4r)$, $f_i^{-1}(p)$ is diffeomorphic to $S^1 / H_p$ for some cyclic group $H_p= \mathbb{Z}_m$. Then for any small disk $D \subset B(o, 4r)$ with $D \cap S_1(X) = \{p\} $, an $m$-fold covering of $f_i^{-1}(D)$ is diffeomorphic to $D \times S^1$. So 
	$$\Omega _i:= f_i^{-1}(B(o, 3r)\setminus B(o, r) )$$ is a Seifert fibered space over orbifold $B(o, 3r)\setminus B(o, r)$.

	Set $\partial _- := f_i^{-1}(\partial B(o, r) )$ and $\partial _+:= f_i^{-1}(\partial B(o, 3r) )$. Perturbing $B(o, 3r)\setminus B(o,r)$ by $B(o, 3r+\varepsilon )\setminus B(o, r-\varepsilon )$ for some $0< \varepsilon \ll 1$ if necessary, we can assume $\partial ( B(o, 3r)\setminus B(o,r) )$ consists of no orbifold singularities such that $\partial \Omega _i= \partial _- \sqcup \partial _+$ is smooth. Since $$
	d(o, f_i(\partial _-) )\leq r,\ d(o, f_i(\partial _+) ) \geq 3r,
	$$ 
	we know for all large enough $i$,
	$$
	d_i(\partial _-, \partial _+) \geq 2r - C (1+r) \varepsilon _i \geq r.
	$$ 

\end{proof}

\begin{prop}\label{dimX=1}
	If $\dim X=1$, then for any fixed number $r>1$, there is a submanifold  $T_i \times (-1, 1)\subset M_i$ with $d_i(T_i \times \{-1\} , T_i \times \{1\} ) \geq r$, where $T_i$ is diffeomorphic to a torus $T^2$ or a Klein bottle $K^2$.
\end{prop}

\begin{proof}
	If $\dim X=1$, then $X=\mathbb{R}$ or $[0,\infty)$. In both cases, we can choose $q\in X$ with $d(o,q)=r+1$ and $B(q, r)\subset X$ consists of regular points. Then by \cite[Theorem 0.12, Theorem 10.1]{Fukaya88}, for large enough $i$, there is a map $$f_i: B(p_i, 2r+2)\to X$$ such that the restriction on $f_i^{-1}(B(q, r) )$ is a fiber bundle, whose fiber is diffeomorphic to a torus or Klein bottle with arbitrary small diameter and $f_i$ is an almost Riemannian submersion. 
		So $$\Omega _i= f_i^{-1}(B(q, r) ) \cong T_i \times (-1,1)$$ is a desired submanifold. Note $\lim_{i\to \infty}d_i(T_i \times \{-1\} , T_i \times \{1\} )= 2r.$

\end{proof}
\section{Proof of the theorem}\label{main-proof}
Recall that a Riemannian band is a Riemannian manifold $(Y,\partial _{\pm})$ with two distinguished disjoint non-empty subsets in the boundary $\partial Y$, denoted by $\partial _-$ and $\partial _+$. A band is called proper if $\partial _{\pm}$ are unions of connected components of $\partial Y$ and $\partial _- \cup \partial _+ = \partial Y$. The width of a Riemannian band $(Y, \partial _{\pm})$ is defined as $d(\partial _-, \partial _+).$ Gromov proved the following width inequalities in \cite{Gromov19} by using $\mu $-bubbles.
\begin{prop}[$\frac{2\pi}{n}$-Inequality]\label{n-dim-width-ineq}
	Let $(Y^n, \partial _{\pm})$ be a proper compact Riemannian band of dimension $n \leq 7$ with scalar curvature $R \geq 1$. If no closed hypersurface in $Y$ which separates $\partial _-$ from $\partial _+$ admits a metric with positive scalar curvature, then $$
	d(\partial _-, \partial _+) \leq \frac{2\pi}{n}\cdot \sqrt{n(n-1)} .
	$$ 
\end{prop}

As a corollary, width inequalities hold for torical bands, see also \cite{Gromov18}.

\begin{prop}\label{torical-band}	
	For a Riemannian torical band $ T^{n-1}\times [-1,1]$ with dimension $n \leq 7$ and scalar curvature $R \geq 1$,
	$$
	d(T^{n-1}\times \{-1\}, T^{n-1}\times \{1\}) \leq \frac{2\pi}{n}\cdot \sqrt{n(n-1)} .
	$$ 
\end{prop}

Note that a closed orientable surface which admits a metric with positive scalar curvature is  diffeomorphic to a $2$-sphere $S^2$. Taking $n=3$, we have the following corollary.

\begin{prop}\label{3-dim-width-ineq}
	Let $(Y^3, \partial _{\pm})$ be a three dimensional proper compact orientable Riemannian band with scalar curvature $R \geq 1$. There exists a positve number $C_0=\frac{2 \sqrt{6} \pi}{3}$ such that if $d(\partial _-, \partial _+) \geq C_0$, then there exist $2$-spheres separating $\partial _-$ from $\partial _+$.
\end{prop}

Now we can prove the main theorem.
\begin{proof}[Proof of Theorem \ref{main-thm}]
	Assume by contradiction that there exists a sequence of complete noncompact three manifolds $(M_i, g_i, p_i)$, satisfying $|K_{g_i}| \leq \Lambda $ and $R_{g_i}\geq 1$, such that $\inj(p_i) \to 0$. Taking a double covering if necessary, we assume $M_i$ are orientable. As in Section 2, up to a subsequence, we assume $(M_i, g_i, p_i)\to (X, d, o)$ in the pointed Gromov-Hausdorff topology, with $\dim X =1$ or $2$.

	If $\dim X=2$, from Proposition \ref{seifert}, we know for a big fixed radius $r>C_0+100$, where $C_0$ is the constant in Proposition \ref{3-dim-width-ineq}, there exists a Seifert fibered space $\Omega_i $ with two disjoint non-empty smooth boundaries $\partial \Omega _i = \partial _- \sqcup \partial _+$ such that
	$$
	d_{g_i}(\partial _-, \partial _+) \geq r.
	$$ 
	Applying Proposition \ref{3-dim-width-ineq} to $(\Omega _i, \partial _{\pm})$, by our choice of $r$, we know there are $2$-spheres $\{S_k^2\} _{k=1}^N$ in $\Omega _i$ separating $\partial _-$ from $\partial _+$. In particular, we have $[\cup _{k=1}^N S_k^2]= [\partial _-] \neq 0$ in $H_2(\Omega _i)$.

	On the other hand, by \cite[Lemma 3.1]{Scott83}, the universal covering of $\Omega _i$ is $\mathbb{R}^3$, so $\Omega_i $ is $S^2$-irreducible and $[\cup _{k=1}^{N}S_k^2]=0$ in $H_2(\Omega _i)$, which is a contradiction. This completes the proof of the case when $\dim X=2$.

If $\dim X=1$, from Proposition \ref{dimX=1}, for any $r>0$, we know there is a torical band $T \times (-1,1)$ with width bigger than $r$ admitting a positive scalar curvature $R_{g} \geq 1$, where $T$ is either a torus or Klein bottle. Up to a double covering of $T\times (-1,1)$, we can assume $T$ is a torus. Choosing $r$ big enough gives a contradiction to Proposition \ref{torical-band}.
	
\end{proof}

\bibliographystyle{alpha}
\bibliography{./math}

\end{document}